\documentclass[12pt]{amsart}
\usepackage{amsmath, amsthm, amscd, amsfonts, amssymb, graphicx, color}
\usepackage[pagebackref=false,colorlinks,linkcolor=blue,citecolor=red]{hyperref}
\usepackage[all]{xy}

\usepackage{pgf,tikz,float}
\usepackage{mathrsfs}
\usetikzlibrary{arrows}

\def\NZQ{\mathbb}               
\def\NN{{\NZQ N}}

\def\frk{\frak}               

\def\Phi{{\frk n}}
\def\Phi{{\frk N}}
\def\KK{{\mathbb K}}
\def\opn#1#2{\def#1{\operatorname{#2}}} 
\opn\chara{char}
\opn\length{\ell}
\opn\pd{pd}
\opn\rk{rk}
\opn\projdim{proj\,dim}
\opn\injdim{inj\,dim}
\opn\rank{rank}
\opn\depth{depth}
\opn\diam{diam}
\opn\grade{grade}
\opn\height{height}
\opn\embdim{emb\,dim}
\opn\codim{codim}

\opn\Tr{Tr}
\opn\bigrank{big\,rank}
\opn\superheight{superheight}
\opn\lcm{lcm}
\opn\trdeg{tr\,deg}
\opn\reg{reg}
\opn\hilb{Hilb}
\opn\hpolynomial{h}
\opn\cdeg{cdeg}
\opn\lreg{lreg}
\opn\ini{in}
\opn\lpd{lpd}
\opn\size{size}
\opn\bigsize{bigsize}
\opn\cosize{cosize}
\opn\bigcosize{bigcosize}
\opn\sdepth{sdepth}
\opn\sreg{sreg}
\opn\link{link}
\opn\fdepth{fdepth}
\opn\lin{lin}
\opn\ini{in}
\opn\div{div}
\opn\Div{Div}
\opn\cl{cl}
\opn\Cl{Cl}
\opn\Spec{Spec}
\opn\Supp{Supp}
\opn\supp{supp}
\opn\Sing{Sing}
\opn\Ass{Ass}
\opn\Min{Min}
\opn\Mon{Mon}
\opn\dstab{dstab}
\opn\astab{astab}
\opn\Syz{Syz}
\opn\Ann{Ann}
\opn\Rad{Rad}
\opn\Soc{Soc}
\opn\Im{Im}
\opn\Ker{Ker}
\opn\Coker{Coker}
\opn\Am{Am}
\opn\Hom{Hom}
\opn\Tor{Tor}
\opn\Ext{Ext}
\opn\End{End}
\opn\Aut{Aut}
\opn\id{id}

\opn\nat{nat}
\opn\pff{pf}
\opn\Pf{Pf}
\opn\GL{GL}
\opn\SL{SL}
\opn\mod{mod}
\opn\ord{ord}
\opn\Gin{Gin}
\opn\Hilb{Hilb}
\opn\sort{sort}
\opn\initial{init}
\opn\ende{end}
\opn\height{height}
\opn\type{type}
\opn\mdeg{mdeg}
\opn\aff{aff}
\opn\con{conv}
\opn\relint{relint}
\opn\st{st}
\opn\lk{lk}
\opn\cn{cn}
\opn\core{core}
\opn\vol{vol}
\opn\link{link}
\opn\star{star}
\opn\lex{lex}
\opn\sign{sign}
\opn\gr{gr}

\def\pot#1#2{#1[\kern-0.28ex[#2]\kern-0.28ex]}

\opn\dirlim{\underrightarrow{\lim}}
\opn\inivlim{\underleftarrow{\lim}}
\let\to=\rightarrow

\def\Implies{\ifmmode\Longrightarrow \else
	\unskip${}\Longrightarrow{}$\ignorespaces\fi}
\def\implies{\ifmmode\Rightarrow \else
	\unskip${}\Rightarrow{}$\ignorespaces\fi}
\def\iff{\ifmmode\Longleftrightarrow \else
	\unskip${}\Longleftrightarrow{}$\ignorespaces\fi}

\let\:=\colon
\newtheorem{Theorem}{Theorem}[section]
\newtheorem{Lemma}[Theorem]{Lemma}
\newtheorem{Corollary}[Theorem]{Corollary}

\newtheorem{Definition}[Theorem]{Definition}

\newtheorem{Conjecture}[Theorem]{Conjecture}

\let\epsilon\varepsilon
\let\kappa=\varkappa
\def\pnt{{\raise0.5mm\hbox{\large\bf.}}}

\begin{document}
	\title{A proof for a conjecture on the regularity of binomial edge ideals}
	\author {M. Rouzbahani Malayeri, S. Saeedi Madani, D. Kiani}
		
	\address{Mohammad Rouzbahani Malayeri, Department of Mathematics and Computer Science, Amirkabir University of Technology (Tehran Polytechnic), Tehran, Iran}
	\email{m.malayeri@aut.ac.ir}
	
		\address{Sara Saeedi Madani, Department of Mathematics and Computer Science, Amirkabir University of Technology (Tehran Polytechnic), Tehran, Iran, and School of Mathematics, Institute for Research in Fundamental Sciences (IPM), Tehran, Iran}
	\email{sarasaeedi@aut.ac.ir}
	
		\address{Dariush Kiani, Department of Mathematics and Computer Science, Amirkabir University of Technology (Tehran Polytechnic), Tehran, Iran, and School of Mathematics, Institute for Research in Fundamental Sciences (IPM), Tehran, Iran}
	\email{dkiani@aut.ac.ir}

	\begin{abstract}
	In this paper we introduce the concept of clique disjoint edge sets in graphs. Then, for a graph $G$, we define the invariant $\eta(G)$ as the maximum size of a clique disjoint edge set in $G$. We show that the regularity of the binomial edge ideal of $G$ is bounded above by $\eta(G)$. This, in particular, settles a conjecture on the regularity of binomial edge ideals in full generality.  
	
	\end{abstract}

	
	
	\maketitle
	
\section{Introduction}\label{introduction}
	Let $G$ be a graph on the vertex set $[n]$ and the edge set $E(G)$. Let also $S=\KK[x_1 ,\ldots ,x_n , y_1 , \ldots , y_n]$ be the polynomial ring over a field $\KK$. Then, the \emph{binomial edge ideal} associated to $G$, denoted by $J_G$, is the ideal in $S$ generated by all the quadratic binomials of the form $f_{ij}=x_{i}y_{j}-x_{j}y_{i}$, where $\{i,j\}\in E(G)$ and $i<j$. This class of ideals were introduced in \cite{HHHKR} and \cite{O}, as a natural generalization of determinantal ideals, as well as the ideals generated by the adjacent $2$-minors of a $2\times n$-matrix of indeterminates.
\par In the meantime, many researchers have studied the algebraic properties and homological invariants of binomial edge ideals. A main goal is to understand how the invariants and properties of the ideal and the underlying graph are related, see e.g. \cite{A, BN, BMS, EHH, ERT, EZ, KS2, KumarS, MM, RSK2, SK, SK1, SK2} for some efforts in this direction.
\par One of the most interesting homological invariants associated to binomial edge ideals that has attracted much attention is the \emph{Castelnuovo-Mumford regularity}, (or regularity for simplicity), namely,
\[
\reg S/J_G=\max \{j-i: \beta_{i,j}(S/J_G)\neq 0\}.
\] 
\par In \cite{SK}, the second and third authors of the present paper,   characterized the graphs $G$ for which $\reg S/J_G=1$. They also gave a characterization of the graphs $G$ with $\reg S/J_G=2$, in \cite{SK2}. Another important result about the regularity of this class of binomial ideals appeared in \cite{MM}, where the authors showed that 
$$\mathcal{L}(G)\leq \reg S/J_G\leq n-1,$$
where $\mathcal{L}(G)$ denotes the sum of the lengths of longest induced paths of connected components of $G$. Recently, the upper bound $n-1$ has been slightly improved in \cite{ERT}. In \cite{MM} the authors additionally conjectured that $\reg S/J_G\leq n-2$, if $G$ is not $P_n$, the path on $n$ vertices.  Later, in \cite{KS2}, this conjecture was proved by the second and the third authors of this paper. On the other hand, in \cite{SK} it was shown that $\reg S/J_G\leq c(G)$, for the so-called \emph{closed} graphs (also known as \emph{proper interval} graphs), where $c(G)$ denotes the number of maximal cliques of $G$. Afterwards, in $2013$, the following conjecture regarding the regularity of binomial edge ideals was posed by the second and third authors of this paper, (see \cite[page~12]{SK1} and  \cite[Conjecture~A]{KS2}).
\begin{Conjecture}\label{conj}
Let $G$ be a graph. Then
$$\reg S/J_G\leq c(G).$$
\end{Conjecture}
\par Recall that a \emph{chordal} graph is a graph with no induced cycle of length greater than $3$. In \cite{EZ}, Ene and Zarojanu verified Conjecture \ref{conj} for a class of chordal graphs, called \emph{block} graphs (i.e. chordal graphs in which any two maximal cliques intersect in at most one vertex). In \cite{JK} the conjecture was proved for the so-called \emph{fan} graphs of complete graphs, another subclass of chordal graphs. Afterwards in \cite{RSK}, and later independently in \cite{Kumar2}, the authors verified Conjecture \ref{conj} for all chordal graphs. Very recently, in \cite{KK} the conjecture was proved for $P_4$-free graphs. 
\par In this paper first we supply a general upper bound for the regularity of binomial edge ideals. This bound indeed is based on a new concept that we call it compatible maps. In fact, such maps are defined from the set of all graphs to the set of non-negative integers that admit some specific properties. We also introduce the notion of clique disjoint edge set in graphs. Then, we associate to each graph $G$, a graphical invariant denoted by $\eta(G)$, which is defined as the maximum size of a clique disjoint edge set in $G$. This enables us to provide a good  combinatorial candidate of a compatible map which, in turn, yields a combinatorial upper bound for the regularity of binomial edge ideals. Then, in particular, we settle Conjecture \ref{conj} in full generality.  Furthermore, we compare some of the known bounds for the regularity of binomial edge ideals in some examples. In particular, we give an infinite family $\{G_n\}_{n=1}^{\infty}$ of graphs with 
$$\lim_{n \to \infty}(c(G_n)-\eta(G_n))=\infty.$$
Finally, a natural question regarding the regularity of binomial edge ideals will be posed.  
\par Throughout the paper, all graphs are assumed to be simple (i.e. with no direction, loops and multiple edges).

\section{Upper bounds for the regularity of binomial edge ideals} 
In this section we first introduce the concept of compatible maps from the set of all graphs to the set of non-negative integers. Then, we investigate about the regularity of binomial edge ideals considering this new concept. We also introduce the concept of clique disjoint edge sets  in graphs to provide a combinatorial compatible map. This, in particular,  enables us to prove Conjecture \ref{conj} in full generality.

\par \medskip In the following, for a graph $G$ and $T\subseteq V(G)$, we use the notation $G-T$, for the induced subgraph of $G$ on the vertex set $V(G)\setminus T$. In particular, if $T=\{v\}$, we use the notation $G-v$ instead of $G-\{v\}$, for simplicity. Moreover, we say that $v$ is a   \emph{free} vertex of $G$, if the induced subgraph of $G$ on the vertex set $N_G(v)$ is a complete graph. Also, we set $\widehat{G}=G-\mathcal{I}s(G)$, where $\mathcal{I}s(G)$ denotes the set of isolated vertices of $G$. Moreover, by $K_t$ we mean the complete graph on $t$ vertices, for every $t\in \NN_0=\NN\cup\{0\}$. 
\par Let $G$ be a graph on $V(G)=[n]$ and $v\in [n]$. Associated to the vertex $v$, there is a graph, denoted by $G_{v}$, with the vertex set $V(G)$ and the edge set 
\[
E(G)\cup \{\{u,w\}: \{u,w\}\subseteq N_{G}(v)\},
\]
where $N_{G}(v)$ denotes the set of neighbours of the vertex $v$ in $G$.
\par Now, in the following definition, we introduce certain maps from the set of all graphs to the set of non-negative integers $\NN_0$. This enables us to obtain a general upper bound for the regularity of binomial edge ideals. 
\begin{Definition}\label{varphi}
\em{
Let $\mathcal{G}$ be the set of all graphs. We call a map $\varphi:\mathcal{G}\longrightarrow \NN_0$, \emph{compatible}, if it satisfies the following conditions:
\begin{enumerate}
\item[{(a)}] $\varphi(\widehat{G})\leq \varphi(G)$, for every $G\in \mathcal{G}$;
\item[{(b)}] if $G=\dot{\cup}_{i=1}^{t}K_{n_i}$, where $n_i\geq 2$ for every $1\leq i\leq t$, then $\varphi(G)\geq t$;
\item[{(c)}] if $G\neq \dot{\cup}_{i=1}^{t}K_{n_i}$, then there exists $v\in V(G)$ such that
\begin{enumerate}
\item[{(1)}] $\varphi(G-v)\leq \varphi(G)$, and
\item[{(2)}] $\varphi(G_v)<\varphi(G)$.
\end{enumerate}
\end{enumerate}
}
\end{Definition}

\par We use the following lemma from \cite{Kumar}. In the following, $iv(G)$ denotes the number of non-free vertices of a graph $G$.
  
\begin{Lemma}{\cite[Lemma~3.4]{Kumar}}\label{simplicial}
Let $G$ be a graph and $v$ be a non-free vertex of $G$. Then, 
$\max\{iv(G_v), iv(G-v), iv(G_{v}-v)\}< iv(G)$.
\end{Lemma}

We also need to fix a notation from \cite{HHHKR} that will be used in the next theorem. Let $G$ be a graph on $[n]$ and $T\subseteq [n]$. Assume that $G_1 , \ldots , G_{c_G(T)}$ are the connected components of $G-T$. Let $\widetilde{G}_1 , \ldots , \widetilde{G}_{c_G(T)}$ be the complete graphs on the vertex sets $V(G_1), \ldots , V(G_{c_G(T)})$, respectively. Now, by $P_T(G)$ we mean the prime ideal 
\[
P_T(G)=(x_i,y_i)_{i\in T} +J_{\widetilde{G}_1}+\cdots +J_{\widetilde{G}_{c_G(T)}},
\]
in the polynomial ring $S$.

Now, we are ready to state our first main theorem that establishes a general upper bound for the regularity of binomial edge ideals. 	
\begin{Theorem}\label{main1}
Let $G$ be a graph on $[n]$ and $\varphi$ be a compatible map. Then 
$$\reg S/J_G\leq \varphi(G).$$
\end{Theorem}	
\begin{proof}
\par We prove the assertion by induction on $iv(G)$. If $iv(G)=0$, then $G$ is a disjoint union of complete graphs. Let $\widehat{G}=\dot{\cup}_{i=1}^{t}K_{n_i}$, where $n_i\geq 2$ for every $1\leq i \leq t$. It is well-known that $\reg S/J_G=\reg \widehat{S}/J_{\widehat{G}}$, where $\widehat{S}=\KK[x_i,y_i:i\in[n]\setminus \mathcal{I}s(G)]$. By \cite[Theorem~2.1]{SK}, we have $\reg \widehat{S}/J_{\widehat{G}}=t$. On the other hand, we have $t\leq \varphi(\widehat{G})\leq \varphi(G)$, by Definition \ref{varphi}, parts $(a)$ and $(b)$. Therefore, in this case the assertion holds. 
\par Now, we assume that $iv(G)>0$. Let $v\in [n]$ be the desired vertex for $\varphi$ in condition $(c)$ in Definition \ref{varphi}. 
\par Let $Q_1=\bigcap\limits_{\substack{T\subseteq [n],\\v\notin T}}P_T(G)$  and $Q_2=\bigcap\limits_{\substack{T\subseteq [n],\\v\in T}}P_T(G)$. We have that $Q_1=J_{G_v}$, $Q_2=(x_v,y_v)+J_{G-v}$ and also $Q_1+Q_2=(x_v,y_v)+J_{{G_v}-v}$, see \cite[Proof~of~Theorem~1.1]{EHH} and  \cite[Proof~of~Theorem~3.5]{RSK}. Therefore, the short exact sequence
\[
0\longrightarrow  \dfrac{S}{J_G}\longrightarrow \dfrac{S}{J_{G_v}}\oplus \dfrac{S_v}{J_{G-v}} \longrightarrow  \dfrac{S_v}{J_{{G_v}-v}}\longrightarrow 0,
\]
is induced, where $S_v=\KK[x_i,y_i:i\in[n]\setminus \{v\}]$. 
\par Now, the well-known regularity lemma implies that
\begin{equation}\label{exact}
\reg S/J_G\leq \max \{\reg S/J_{G_v}, \reg S_v/J_{G-v}, \reg S_v/J_{G_v-v}+1\}.
\end{equation}
By Lemma \ref{simplicial} and by the induction hypothesis, we get
\begin{equation}\label{first}
\reg S/J_{G_v}\leq \varphi(G_v)<\varphi (G),
\end{equation}
and
\begin{equation}\label{second}
\reg S_v/J_{G-v}\leq \varphi(G-v)\leq \varphi(G).
\end{equation}
Since $G_{v}-v$ is an induced subgraph of $G_v$, by \cite[Proposition~8,~part~(b)]{SK1} we have $\reg S_v/J_{G_v-v}\leq \reg S/J_{G_v}$, and hence by \eqref{first} we get 
\begin{equation}\label{third}
\reg S_v/J_{G_v-v}<\varphi(G).
\end{equation}
Therefore, the result follows by \eqref{exact}, \eqref{first}, \eqref{second} and \eqref{third}.
\end{proof}	
	
\par \medskip Next, we are going to provide a combinatorial compatible map. For this purpose, we assign a graphical invariant to a graph $G$, denoted by $\eta(G)$.
\begin{Definition}
\em
{Let $G$ be a graph and $\mathcal{H}\subseteq E(G)$ with the property that no two elements of $\mathcal{H}$ belong to a clique of $G$. Then, we call the set $\mathcal{H}$, a \emph{clique disjoint edge set} in $G$.
} 
\end{Definition}	

\par Moreover, we set 
$$\eta(G):=\max\{|\mathcal{H}|: \mathcal{H}~\mathrm{is~a~clique~disjoint~edge~set~in}~G\}.$$

\par Now, in the next theorem, we provide a compatible map given by $\eta(G)$.   
\begin{Theorem}\label{main2}
The map $\eta:\mathcal{G}\longrightarrow \NN_0$ is compatible.
\end{Theorem}	
\begin{proof}
Let $G\in \mathcal{G}$. It is clear that $\eta(\widehat{G})=\eta(G)$. Moreover, if $G=\dot{\cup}_{i=1}^{t}K_{n_i}$, where $n_i\geq 2$ for every $1\leq i\leq t$, then we have that $\eta(G)=t$. Therefore, it is enough to see that $\eta$ satisfies condition $(c)$ of Definition \ref{varphi}.
\par Assume that $G$ is not a disjoint union of complete graphs. Therefore, there exists $v\in V(G)$ such that $v$ is not a free vertex of $G$. We first observe that $\eta(G-v)\leq \eta(G)$. This indeed follows from the fact that every clique disjoint edge set in $G-v$ is also a clique disjoint edge set in $G$, since $G-v$ is an induced subgraph of $G$.
\par Now assume that $\eta(G_v)=|\mathcal{H}|$, where $\mathcal{H}=\{e_1,\ldots,e_{\eta(G_v)}\}$ is a clique disjoint edge set in $G_v$. We consider the following cases:
\par First assume that $v\in \bigcup \limits_{\substack{e_i \in \mathcal{H}}}e_i$. Without loss of generality assume that $v\in e_1$. Note that $v\notin e_j$, for every $2\leq j \leq \eta(G_v)$. Indeed, assume on the contrary that $v\in e_j$, for some $2\leq j \leq \eta(G_v)$. Then, the edges $e_1$ and $e_j$ belong to a clique of $G_v$, a contradiction.
\par On the other hand, we have that $\mathcal{H}\setminus \{e_1\}\subseteq E(G)$. Indeed, otherwise assume that $e_j=\{u_j,w_j\}\notin E(G)$ for some  $2\leq j \leq \eta(G_v)$. Therefore, we have that $\{v,u_j\}\in E(G)$ and $\{v,w_j\}\in E(G)$. This implies that $e_1$ and $e_j$ belong to a clique of $G_v$, which is a contradiction. Also, since $v$ is not a free vertex of $G$, there exist vertices $\alpha$ and $\beta$ of $G$ such that $\{\alpha,\beta\}\subseteq N_G(v)$ and $\{\alpha,\beta\}\notin E(G)$. Now, it is observed that $\mathcal{H'}=\{\{v,\alpha\},\{v,\beta\},e_2,\ldots,e_{\eta(G_v)}\}$ is a clique disjoint edge set in $G$. Indeed, otherwise assume that either $\{v,\alpha\}$ and $e_j$ or $\{v,\beta\}$ and $e_{j'}$ belong to a clique of $G$ for some $2\leq j,j'\leq \eta(G_v)$. Then, $e_1$ and $e_j$ or $e_1$ and $e_{j'}$ belong to a clique of $G_v$, a contradiction. Also, we have that $\{\{v,\alpha\},\{v,\beta\}\}\cap \{e_2,\ldots,e_{\eta(G_v)}\}=\emptyset$, since $v\notin e_j$, for every $2\leq j \leq \eta(G_v)$. This implies that $|\mathcal{H'}|=\eta(G_v)+1$. Therefore, in this case we have that $\eta(G)\geq \eta(G_v)+1$, as desired.
\par Next assume that $v\notin \bigcup \limits_{\substack{e_i \in \mathcal{H}}}e_i$. Now, if there exists $j=1,\ldots,\eta(G_v)$ with $e_j=\{u_j,w_j\}\notin E(G)$, then with the same argument as used in the previous case, one could see that $\mathcal{H'}=(\mathcal{H}\setminus\{e_j\})\cup\{\{v,u_j\},\{v,w_j\}\}$ is a clique disjoint edge set in $G$ with $|\mathcal{H'}|=\eta(G_v)+1$. This implies that $\eta(G)\geq \eta(G_v)+1$. So, we may assume that $\mathcal{H}\subseteq E(G)$. Since  $v$ is not a free vertex of $G$, there exist vertices $\alpha,\beta \in N_G(v)$ such that $\{\alpha,\beta\}\notin E(G)$.  Notice that if for each $1\leq i \leq \eta(G_v)$ the edges $e_i$ and $\{v,\alpha\}$ do not belong to a clique of $G$, then $\mathcal{H}_{\alpha}=\mathcal{H}\cup \{\{v,\alpha\}\}$ is a clique disjoint edge set in $G$. Similarly, if for each $1\leq i \leq \eta(G_v)$ the edges $e_i$ and $\{v,\beta\}$ do not belong to a clique of $G$, then $\mathcal{H}_{\beta}=\mathcal{H}\cup \{\{v,\beta\}\}$ is a clique disjoint edge set in $G$. Thus, we get the desired result. Therefore, we assume that $e_i$ and $\{v,\alpha\}$ belong to a clique of $G$ and also $e_j$ and $\{v,\beta\}$ belong to a clique of $G$ for some $e_i,e_j\in \mathcal{H}$. This implies that $i=j$, otherwise $e_i$ and $e_j$ belong to a clique of $G_v$, which is a contradiction. Now, it is seen that 
$$\mathcal{H''}=(\mathcal{H}\setminus\{e_i\})\cup\{\{v,\alpha\},\{v,\beta\}\}$$
is a clique disjoint edge set in $G$ with $|\mathcal{H''}|=\eta(G_v)+1$, and hence $\eta(G_v)<\eta(G)$, as desired.
\end{proof}
\par Now, combining Theorem \ref{main1} and Theorem \ref{main2} we obtain: 
\begin{Corollary}\label{eta}
Let $G$ be a graph on $[n]$. Then 
$$\reg S/J_G\leq \eta(G).$$
\end{Corollary}

We would like to remark that the above upper bound for the regularity could be sharp. For instance, let $G$ be the graph illustrated in Figure \ref{sharp1}. Then, $\eta(G)=4$. Also, $\reg S/J_G=4$ by  \cite[Proposition~3.8]{KS}.

\begin{figure}[H]
\centering
\begin{tikzpicture}[scale=1.1,line cap=round,line join=round,>=triangle 45,x=1.0cm,y=1.0cm]
\draw (9.,2.)-- (9.,1.);
\draw (8.,0.)-- (9.,1.);
\draw (9.,1.)-- (10.,0.);
\draw (10.,0.)-- (8.,0.);
\draw (8.,-1.)-- (8.,0.);
\draw (10.,0.)-- (10.,-1.);
\begin{scriptsize}
\draw [fill=blue] (9.,1.) circle (1.5pt);
\draw [fill=blue] (8.,0.) circle (1.5pt);
\draw [fill=blue] (10.,0.) circle (1.5pt);
\draw [fill=blue] (8.,-1.) circle (1.5pt);
\draw [fill=blue] (10.,-1.) circle (1.5pt);
\draw [fill=blue] (9.,2.) circle (1.5pt);
\end{scriptsize}
\end{tikzpicture}
\vspace{3mm}
\caption{A graph $G$ with $\reg S/J_G=\eta(G)=4$.}
\label{sharp1}
\end{figure}
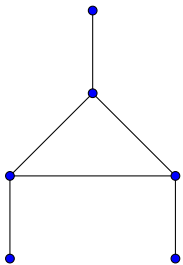

On the other hand, there are graphs $G$ for which $\reg S/J_G<\eta(G)$. For example, let $G$ be the closed graph illustrated in Figure \ref{closed}  with $\eta(G)=4$ and $\mathcal{L}(G)=3$, where $\mathcal{L}(G)$ is  the length of a longest induced path of $G$. Then, by \cite[Theorem~2.2]{EZ} we have $\reg S/J_G=3$. In addition, $G-v$ is a closed graph too, with $\reg S_v/J_{G-v}=\mathcal{L}(G-v)=\eta(G-v)=3$, for every vertex $v$ of $G$.

\begin{figure}[H]
\centering
\begin{tikzpicture}[scale=1.5,line cap=round,line join=round,>=triangle 45,x=1.0cm,y=1.0cm]
\draw (2.,0.)-- (1.,0.);
\draw (0.,1.)-- (-1.,0.);
\draw (-1.,0.)-- (0.,0.);
\draw (0.,1.)-- (1.,1.);
\draw (1.,1.)-- (2.,0.);
\draw (0.,0.)-- (0.,1.);
\draw (0.,0.)-- (1.,1.);
\draw (0.,0.)-- (1.,0.);
\draw (1.,0.)-- (1.,1.);
\begin{scriptsize}
\draw [fill=blue] (0.,1.) circle (1.1pt);
\draw [fill=blue] (1.,1.) circle (1.1pt);
\draw [fill=blue] (2.,0.) circle (1.1pt);
\draw [fill=blue] (1.,0.) circle (1.1pt);
\draw [fill=blue] (0.,0.) circle (1.1pt);
\draw [fill=blue] (-1.,0.) circle (1.1pt);
\end{scriptsize}
\end{tikzpicture}
\vspace{3mm}
\caption{A closed graph $G$ with $\reg S/J_G=\mathcal{L}(G)<\eta(G)$.}
\label{closed}
\end{figure}
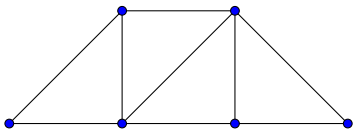

\par \medskip It is clear that $\eta(G)\leq c(G)$, for every graph $G$. Therefore, as a consequence of Corollary \ref{eta} we get the following upper bound for the regularity of binomial edge ideals, which settles Conjecture \ref{conj} affirmatively.

\begin{Corollary}\label{cg}
Let $G$ be a graph on $[n]$. Then 
$$\reg S/J_G\leq c(G).$$
\end{Corollary}	
	
Note that there are some graphs $G$ for which $\reg S/J_G$ attains the upper bound $\eta(G)$ with $\eta(G)<c(G)$. For example, the graph $G_1$  depicted in Figure~\ref{family} has this property. Indeed, we have that $c(G_1)=4$. Moreover, it is easily seen that $\mathcal{L}(G_1)=\eta(G_1)=3$.  Therefore, we have $\reg S/J_G=3$, since the upper bound given in Corollary \ref{eta} coincides with the lower bound given in \cite[Theorem~1.1]{MM}.

\par Furthermore, we would like to construct an infinite family $\{G_n\}_{n=1}^{\infty}$ of graphs to show that the difference between the upper bounds $\eta(G_n)$ and $c(G_n)$ could be big enough for sufficiently large values of $n$. For this aim, let $G_1$ be the left side graph depicted in Figure \ref{family}. We follow the pictorial pattern illuminated in Figure~\ref{family} to obtain the graph $G_n$ for every $n\geq 2$, by replacing any triangle of $G_{n-1}$ by a copy of $G_1$. It is observed that $c(G_n)=4^n$ and $\eta(G_n)\leq 3\times 4^{n-1}$. Indeed, the latter inequality follows from the facts that $\eta(G_1)=3$ and also the graph $G_n$ is covered by $4^{n-1}$ copies of $G_1$, for every $n\geq 2$. Thus, $$\lim_{n \to \infty}(c(G_n)-\eta(G_n))=\infty.$$

\begin{figure}[H]
\definecolor{qqqqff}{rgb}{0.,0.,1.}
\centering
\begin{tikzpicture}[scale=1.1,line cap=round,line join=round,>=triangle 45,x=1.0cm,y=1.0cm]
\draw [line width=1.3pt] (1.5,2.6)-- (0.,0.);
\draw [line width=1.3pt] (1.5,2.6)-- (3.,0.);
\draw [line width=1.3pt] (0.,0.)-- (3.,0.);
\draw [line width=1.3pt] (0.7524972253052165,1.3043285238623752)-- (2.24850166481687,1.302597114317425);
\draw [line width=1.3pt] (1.5,0.)-- (0.7524972253052165,1.3043285238623752);
\draw [line width=1.3pt] (1.5,0.)-- (2.24850166481687,1.302597114317425);
\draw (1.874250832408435,1.9512985571587125)-- (1.1262486126526081,1.9521642619311876);
\draw (1.5004994450610432,1.3034628190899)-- (1.874250832408435,1.9512985571587125);
\draw (1.1262486126526081,1.9521642619311876)-- (1.5004994450610432,1.3034628190899);
\draw (1.874250832408435,0.6512985571587125)-- (1.1262486126526081,0.6521642619311876);
\draw (1.1262486126526081,0.6521642619311876)-- (1.5004994450610432,1.3034628190899);
\draw (1.874250832408435,0.6512985571587125)-- (1.5004994450610432,1.3034628190899);
\draw (1.1262486126526081,0.6521642619311876)-- (0.3762486126526082,0.6521642619311876);
\draw (0.3762486126526082,0.6521642619311876)-- (0.75,0.);
\draw (0.75,0.)-- (1.1262486126526081,0.6521642619311876);
\draw (1.874250832408435,0.6512985571587125)-- (2.624250832408435,0.6512985571587125);
\draw (2.25,0.)-- (2.624250832408435,0.6512985571587125);
\draw (2.25,0.)-- (1.874250832408435,0.6512985571587125);
\draw (-1.,0.)-- (-2.5,0.);
\draw (-2.5,0.)-- (-4.,0.);
\draw (-3.25,1.3)-- (-4.,0.);
\draw (-2.5,2.6)-- (-3.25,1.3);
\draw (-2.5,2.6)-- (-1.75,1.3);
\draw (-1.75,1.3)-- (-1.,0.);
\draw (-3.25,1.3)-- (-1.75,1.3);
\draw (-2.5,0.)-- (-1.75,1.3);
\draw (-2.5,0.)-- (-3.25,1.3);
\draw [line width=1.3pt] (7.,0.)-- (6.625,0.);
\draw [line width=1.3pt] (6.625,0.)-- (6.25,0.);
\draw [line width=1.3pt] (5.875,0.)-- (6.25,0.);
\draw [line width=1.3pt] (5.875,0.)-- (5.5,0.);
\draw [line width=1.3pt] (5.5,0.)-- (5.125,0.);
\draw [line width=1.3pt] (5.125,0.)-- (4.75,0.);
\draw [line width=1.3pt] (4.75,0.)-- (4.375,0.);
\draw [line width=1.3pt] (4.375,0.)-- (4.,0.);
\draw [line width=1.3pt] (5.5,2.6)-- (5.3125,2.275);
\draw [line width=1.3pt] (5.125,1.95)-- (5.3125,2.275);
\draw [line width=1.3pt] (5.125,1.95)-- (4.9375,1.625);
\draw [line width=1.3pt] (4.9375,1.625)-- (4.75,1.3);
\draw [line width=1.3pt] (4.75,1.3)-- (4.5625,0.975);
\draw [line width=1.3pt] (4.5625,0.975)-- (4.375,0.65);
\draw (4.375,0.65)-- (4.1875,0.325);
\draw [line width=1.3pt] (4.375,0.65)-- (4.,0.);
\draw [line width=1.3pt] (5.5,2.6)-- (5.6875,2.275);
\draw (5.3125,2.275)-- (5.6875,2.275);
\draw (5.5,1.95)-- (5.3125,2.275);
\draw (5.5,1.95)-- (5.6875,2.275);
\draw [line width=1.3pt] (5.875,1.95)-- (5.6875,2.275);
\draw [line width=1.3pt] (5.875,1.95)-- (5.5,1.95);
\draw [line width=1.3pt] (5.125,1.95)-- (5.5,1.95);
\draw [line width=1.3pt] (7.,0.)-- (6.8125,0.325);
\draw [line width=1.3pt] (6.8125,0.325)-- (6.625,0.65);
\draw [line width=1.3pt] (6.625,0.65)-- (6.4375,0.975);
\draw [line width=1.3pt] (6.4375,0.975)-- (6.25,1.3);
\draw [line width=1.3pt] (6.25,1.3)-- (6.0625,1.625);
\draw [line width=1.3pt] (6.0625,1.625)-- (5.875,1.95);
\draw (6.0625,1.625)-- (5.6875,1.625);
\draw (5.3125,1.625)-- (5.6875,1.625);
\draw (4.375,0.)-- (4.1875,0.325);
\draw (4.1875,0.325)-- (4.5625,0.325);
\draw (4.5625,0.325)-- (4.9375,0.325);
\draw (4.9375,0.325)-- (5.3125,0.325);
\draw (5.6875,0.325)-- (6.0625,0.325);
\draw (6.0625,0.325)-- (6.4375,0.325);
\draw (6.4375,0.325)-- (6.8125,0.325);
\draw [line width=1.3pt] (6.625,0.65)-- (6.25,0.65);
\draw [line width=1.3pt] (5.5,0.65)-- (5.125,0.65);
\draw [line width=1.3pt] (5.125,0.65)-- (4.75,0.65);
\draw [line width=1.3pt] (4.75,0.65)-- (4.375,0.65);
\draw (4.5625,0.975)-- (4.9375,0.975);
\draw (4.9375,0.975)-- (5.3125,0.975);
\draw (5.3125,0.975)-- (5.6875,0.975);
\draw (6.0625,0.975)-- (5.6875,0.975);
\draw (6.4375,0.975)-- (6.0625,0.975);
\draw [line width=1.3pt] (6.25,1.3)-- (5.875,1.3);
\draw [line width=1.3pt] (5.875,1.3)-- (5.5,1.3);
\draw [line width=1.3pt] (5.5,1.3)-- (5.125,1.3);
\draw [line width=1.3pt] (5.125,1.3)-- (4.75,1.3);
\draw [line width=1.3pt] (5.875,1.95)-- (5.6875,1.625);
\draw (5.3125,1.625)-- (5.5,1.95);
\draw (5.6875,1.625)-- (5.5,1.95);
\draw (5.125,1.3)-- (5.3125,1.625);
\draw (4.9375,0.975)-- (5.125,1.3);
\draw (4.5625,0.325)-- (4.75,0.65);
\draw (4.375,0.)-- (4.5625,0.325);
\draw [line width=1.3pt] (4.5625,0.325)-- (4.375,0.65);
\draw [line width=1.3pt] (4.75,0.)-- (4.5625,0.325);
\draw (4.75,0.65)-- (4.5625,0.975);
\draw (6.625,0.)-- (6.8125,0.325);
\draw (6.625,0.)-- (6.4375,0.325);
\draw [line width=1.3pt] (6.25,0.)-- (6.4375,0.325);
\draw [line width=1.3pt] (6.25,0.)-- (6.0625,0.325);
\draw (5.875,0.)-- (6.0625,0.325);
\draw [line width=1.3pt] (5.6875,0.325)-- (5.5,0.);
\draw (5.875,0.)-- (5.6875,0.325);
\draw [line width=1.3pt] (5.5,0.)-- (5.3125,0.325);
\draw (5.125,0.)-- (5.3125,0.325);
\draw (5.125,0.)-- (4.9375,0.325);
\draw [line width=1.3pt] (4.75,0.)-- (4.9375,0.325);
\draw (4.9375,0.325)-- (4.75,0.65);
\draw [line width=1.3pt] (4.9375,0.325)-- (5.125,0.65);
\draw [line width=1.3pt] (5.3125,0.325)-- (5.125,0.65);
\draw (5.3125,0.325)-- (5.5,0.65);
\draw (5.6875,0.325)-- (5.5,0.65);
\draw [line width=1.3pt] (5.6875,0.325)-- (5.875,0.65);
\draw [line width=1.3pt] (6.0625,0.325)-- (5.875,0.65);
\draw (6.0625,0.325)-- (6.25,0.65);
\draw (6.4375,0.325)-- (6.25,0.65);
\draw [line width=1.3pt] (6.4375,0.325)-- (6.625,0.65);
\draw (6.25,0.65)-- (6.4375,0.975);
\draw [line width=1.3pt] (6.25,0.65)-- (5.875,0.65);
\draw (6.25,0.65)-- (6.0625,0.975);
\draw [line width=1.3pt] (5.875,0.65)-- (6.0625,0.975);
\draw [line width=1.3pt] (5.875,0.65)-- (5.5,0.65);
\draw [line width=1.3pt] (5.875,0.65)-- (5.6875,0.975);
\draw [line width=1.3pt] (5.125,0.65)-- (4.9375,0.975);
\draw (5.5,0.65)-- (5.6875,0.975);
\draw (5.5,0.65)-- (5.3125,0.975);
\draw [line width=1.3pt] (5.125,0.65)-- (5.3125,0.975);
\draw (5.3125,0.975)-- (5.125,1.3);
\draw [line width=1.3pt] (4.9375,0.975)-- (4.75,1.3);
\draw [line width=1.3pt] (5.3125,0.975)-- (5.5,1.3);
\draw [line width=1.3pt] (5.6875,0.975)-- (5.5,1.3);
\draw (5.6875,0.975)-- (5.875,1.3);
\draw [line width=1.3pt] (6.0625,0.975)-- (6.25,1.3);
\draw (6.0625,0.975)-- (5.875,1.3);
\draw [line width=1.3pt] (5.5,1.3)-- (5.3125,1.625);
\draw [line width=1.3pt] (5.5,1.3)-- (5.6875,1.625);
\draw (5.875,1.3)-- (5.6875,1.625);
\draw (5.875,1.3)-- (6.0625,1.625);
\draw (4.9375,0.975)-- (4.75,0.65);
\draw (5.125,1.3)-- (4.9375,1.625);
\draw (4.9375,1.625)-- (5.3125,1.625);
\draw [line width=1.3pt] (5.125,1.95)-- (5.3125,1.625);
\draw (-2.761728395061727,-0.13098765432097778) node[anchor=north west] {$G_1$};
\draw (1.2982716049382772,-0.15098765432097788) node[anchor=north west] {$G_2$};
\draw (5.278271604938282,-0.09098765432097757) node[anchor=north west] {$G_3$};
\draw (5.3125,0.325)-- (5.6875,0.325);
\begin{scriptsize}
\draw [fill=qqqqff] (0.,0.) circle (1.5pt);
\draw [fill=qqqqff] (3.,0.) circle (1.5pt);
\draw [fill=qqqqff] (1.5,2.6) circle (1.5pt);
\draw [fill=qqqqff] (0.7524972253052165,1.3043285238623752) circle (1.5pt);
\draw [fill=qqqqff] (2.24850166481687,1.302597114317425) circle (1.5pt);
\draw [fill=qqqqff] (1.5,0.) circle (1.5pt);
\draw [fill=qqqqff] (1.1262486126526081,1.9521642619311876) circle (1.5pt);
\draw [fill=qqqqff] (1.874250832408435,1.9512985571587125) circle (1.5pt);
\draw [fill=qqqqff] (1.5004994450610432,1.3034628190899) circle (1.5pt);
\draw [fill=qqqqff] (0.3762486126526082,0.6521642619311876) circle (1.5pt);
\draw [fill=qqqqff] (1.1262486126526081,0.6521642619311876) circle (1.5pt);
\draw [fill=qqqqff] (1.874250832408435,0.6512985571587125) circle (1.5pt);
\draw [fill=qqqqff] (2.25,0.) circle (1.5pt);
\draw [fill=qqqqff] (0.75,0.) circle (1.5pt);
\draw [fill=qqqqff] (2.624250832408435,0.6512985571587125) circle (1.5pt);
\draw [fill=qqqqff] (-1.,0.) circle (1.5pt);
\draw [fill=qqqqff] (-4.,0.) circle (1.5pt);
\draw [fill=qqqqff] (-2.5,0.) circle (1.5pt);
\draw [fill=qqqqff] (-2.5,2.6) circle (1.5pt);
\draw [fill=qqqqff] (-3.25,1.3) circle (1.5pt);
\draw [fill=qqqqff] (-1.75,1.3) circle (1.5pt);
\draw [fill=qqqqff] (4.,0.) circle (1.5pt);
\draw [fill=qqqqff] (7.,0.) circle (1.5pt);
\draw [fill=qqqqff] (5.5,0.) circle (1.5pt);
\draw [fill=qqqqff] (5.5,2.6) circle (1.5pt);
\draw [fill=qqqqff] (4.75,1.3) circle (1.5pt);
\draw [fill=qqqqff] (6.25,1.3) circle (1.5pt);
\draw [fill=qqqqff] (4.75,0.) circle (1.5pt);
\draw [fill=qqqqff] (6.25,0.) circle (1.5pt);
\draw [fill=qqqqff] (4.375,0.65) circle (1.5pt);
\draw [fill=qqqqff] (5.125,1.95) circle (1.5pt);
\draw [fill=qqqqff] (5.875,1.95) circle (1.5pt);
\draw [fill=qqqqff] (6.625,0.65) circle (1.5pt);
\draw [fill=qqqqff] (5.125,0.65) circle (1.5pt);
\draw [fill=qqqqff] (5.875,0.65) circle (1.5pt);
\draw [fill=qqqqff] (5.5,1.3) circle (1.5pt);
\draw [fill=qqqqff] (4.375,0.) circle (1.5pt);
\draw [fill=qqqqff] (4.1875,0.325) circle (1.5pt);
\draw [fill=qqqqff] (4.5625,0.325) circle (1.5pt);
\draw [fill=qqqqff] (4.9375,0.325) circle (1.5pt);
\draw [fill=qqqqff] (5.3125,0.325) circle (1.5pt);
\draw [fill=qqqqff] (5.125,0.) circle (1.5pt);
\draw [fill=qqqqff] (5.6875,0.325) circle (1.5pt);
\draw [fill=qqqqff] (5.875,0.) circle (1.5pt);
\draw [fill=qqqqff] (6.0625,0.325) circle (1.5pt);
\draw [fill=qqqqff] (6.4375,0.325) circle (1.5pt);
\draw [fill=qqqqff] (6.625,0.) circle (1.5pt);
\draw [fill=qqqqff] (6.8125,0.325) circle (1.5pt);
\draw [fill=qqqqff] (1.5,0.) circle (1.5pt);
\draw [fill=qqqqff] (4.5625,0.975) circle (1.5pt);
\draw [fill=qqqqff] (4.75,0.65) circle (1.5pt);
\draw [fill=qqqqff] (5.5,0.65) circle (1.5pt);
\draw [fill=qqqqff] (6.25,0.65) circle (1.5pt);
\draw [fill=qqqqff] (4.9375,0.975) circle (1.5pt);
\draw [fill=qqqqff] (4.9375,1.625) circle (1.5pt);
\draw [fill=qqqqff] (5.3125,2.275) circle (1.5pt);
\draw [fill=qqqqff] (6.4375,0.975) circle (1.5pt);
\draw [fill=qqqqff] (6.0625,1.625) circle (1.5pt);
\draw [fill=qqqqff] (5.6875,2.275) circle (1.5pt);
\draw [fill=qqqqff] (5.5,1.95) circle (1.5pt);
\draw [fill=qqqqff] (5.3125,0.975) circle (1.5pt);
\draw [fill=qqqqff] (5.125,1.3) circle (1.5pt);
\draw [fill=qqqqff] (5.3125,1.625) circle (1.5pt);
\draw [fill=qqqqff] (5.6875,1.625) circle (1.5pt);
\draw [fill=qqqqff] (5.875,1.3) circle (1.5pt);
\draw [fill=qqqqff] (6.0625,0.975) circle (1.5pt);
\draw [fill=qqqqff] (5.6875,0.975) circle (1.5pt);
\end{scriptsize}
\end{tikzpicture}
\vspace{3mm}
\caption{The infinite family $\{G_n\}_{n=1}^{\infty}$ of graphs with $\lim_{n \to \infty}(c(G_n)-\eta(G_n))=\infty$.}
\label{family}
\end{figure}
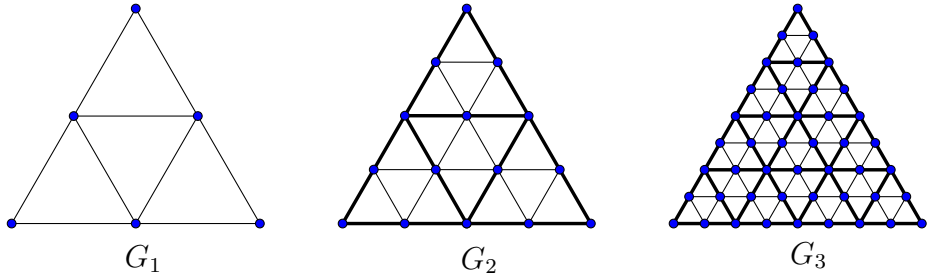
	
\par \medskip Finally, we would like to end this paper with asking a natural question if there is an explicit combinatorial characterization of graphs $G$ with $\mathcal{L}(G)=\eta(G)$. Notice that finding such characterization yields a precise formula for the regularity of the desired class of graphs. It is worth mentioning here that a characterization of chordal graphs $G$ with $\mathcal{L}(G)=c(G)$ was given in \cite[Theorem~4.2]{RSK}. Such graphs are called \emph{strong interval} graphs, which is clear that they satisfy the equality $\mathcal{L}(G)=\eta(G)$ as well.

\vspace{1cm}
\par \textbf{Acknowledgments:} The authors would like to thank the Institute for Research in Fundamental Sciences (IPM) for financial support. The research of the second author was in part supported by a grant from IPM (No. 99130113). The research of the third author was in part supported by a grant from IPM (No. 99050211).


\begin{thebibliography}{99}


		
\bibitem{A} J. \`Alvarez Montaner, {\em Local cohomology of binomial edge ideals and their generic initial ideals}, Collect. Math.  (2019),  https://doi.org/10.1007/s13348-019-00268-z.

\bibitem{BN} A. Banerjee, L. L. N\'u\~nez-Betancourt, {\em  Graph connectivity and binomial edge ideals}, 
Proc. Amer. Math. Soc. 145 (2017), 487-499.
			
\bibitem{BMS} D. Bolognini, A. Macchia, F. Strazzanti, {\em Binomial edge ideals of bipartite graphs}, 
European J. Combin. 70 (2018), 1-25.			
			
\bibitem{EHH} V. Ene, J. Herzog, T. Hibi, {\em Cohen-Macaulay binomial edge ideals}, Nagoya Math. J. 204 (2011), 57-68.
	
\bibitem{ERT} V. Ene, G. Rinaldo, N. Terai, {\em Licci binomial edge ideals}, J. Combin. Theory Ser. A. 175 (2020), 105278, 23 pp.

\bibitem{EZ} V. Ene, A. Zarojanu, {\em On the regularity of binomial edge ideals}, Math. Nachr. 288(1) (2015), 19-24.
		
		
			
\bibitem{HHHKR} J. Herzog, T. Hibi, F. Hreinsd{\'o}ttir, T. Kahle, J. Rauh, {\em Binomial edge ideals and conditional independence statements}, Adv. Appl. Math. 45 (2010), 317-333.
		

		



\bibitem{JK} A. V. Jayanthan, A. Kumar, {\em Regularity of binomial edge ideals of Cohen-Macaulay bipartite graphs}, Comm. Algebra. 47 (2019), 4797-4805.
		

\bibitem{KK} T. Kahle, J. Kr{\"u}semann, {\em Binomial edge ideals of cographs}, (2019), arXiv:1906.05510.


\bibitem{KS} D. Kiani, S. Saeedi Madani, {\em Binomial edge ideals with pure resolutions}, Collect. Math. 65 (2014), 331-340.


\bibitem{KS2} D. Kiani, S. Saeedi Madani, {\em The Castelnuovo-Mumford regularity of binomial edge ideals}, J. Combin. Theory Ser. A. 139 (2016), 80-86.


\bibitem{Kumar} A. Kumar, {\em Regularity bound of generalized binomial edge ideal of graphs}, J. Algebra. 546 (2020), 357-369.

\bibitem{Kumar2} A. Kumar, {\em Binomial edge ideals and bounds for their regularity}, J. Algebraic Combin. (2020), to appear.

\bibitem{KumarS} A. Kumar, R. Sarkar, {\em Depth and extremal Betti number of binomial edge ideals}, Math. Nachr. (2019), to appear. 
		

\bibitem{MM} K. Matsuda, S. Murai, {\em Regularity bounds for binomial edge ideals}, J. Commutative Algebra. 5(1) (2013), 141-149.

\bibitem{O} M. Ohtani, {\em Graphs and ideals generated by some 2-minors}, Comm. Algebra. 39 (2011), 905-917.
		

		
\bibitem{RSK} M. Rouzbahani Malayeri, S. Saeedi Madani, D. Kiani, {\em Regularity of binomial edge ideals of chordal graphs}, (2018) arXiv:1810.03119v1, to appear in Collect. Math.
		
\bibitem{RSK2} M. Rouzbahani Malayeri, S. Saeedi Madani, D. Kiani, {\em Binomial edge ideals of small depth}, J. Algebra.  (2020), to appear. 		
		
		
		
\bibitem{SK} S. Saeedi Madani, D. Kiani, {\em Binomial edge ideals of graphs}, Electronic J. Combin. 19(2) (2012), $\sharp$ P44.
		
\bibitem{SK1} S. Saeedi Madani, D. Kiani, {\em On the binomial edge ideal of a pair of graphs}, Electronic J. of Combinatorics. 20(1) (2013), $\sharp$ P48.
		
\bibitem{SK2} S. Saeedi Madani, D. Kiani, {\em Binomial edge ideals of regularity 3}, J. Algebra. 515 (2018), 157-172.
	

		
		
	\end{thebibliography}
\end{document}